\documentclass[12pt,a4paper]{amsart}
\usepackage{amsmath}
\usepackage{amssymb,amsthm,amscd}

\usepackage{fullpage}

\usepackage{enumerate}

\usepackage[
  backend=biber,
  style=alphabetic,
  sorting=nyt,
  giveninits=true,
  maxbibnames=99
]{biblatex}
\AtBeginBibliography{\sloppy}
\AtEveryBibitem{\clearfield{month}\clearfield{day}}
\DeclareNameAlias{sortname}{given-family}
\DeclareDelimFormat[bib]{multinamedelim}{\addcomma\space}
\DeclareDelimFormat[bib]{finalnamedelim}{\addcomma\space}
\DeclareDelimFormat[bib]{nametitledelim}{\addcomma\space}
\DeclareFieldFormat[article]{title}{\mkbibemph{#1}}
\DeclareFieldFormat{journaltitle}{#1}
\DeclareFieldFormat[article]{volume}{\mkbibbold{#1}}
\DeclareFieldFormat[article]{pages}{#1}
\DeclareFieldFormat{doi}{\url{https://doi.org/#1}}
\DeclareFieldFormat{url}{\url{#1}}

\renewbibmacro{in:}{}
\renewbibmacro*{journal+issuetitle}{%
  \usebibmacro{journal}%
  \setunit*{\addspace}%
  \printfield{volume}%
  \setunit{\addspace}%
  \printtext[parens]{\printfield{year}}%
  \iffieldundef{number}
    {}
    {\setunit{\addcomma\space}%
     \printtext{no.\addnbspace}%
     \printfield{number}}%
  \setunit{\addcomma\space}%
  \printfield{eid}%
  \newunit}

\usepackage[unicode]{hyperref}

\renewcommand{\k}{\Bbbk}
\newcommand{\ZZ}{\mathbb Z}

\newcommand{\Ga}{\mathbb G_{\mathrm a}}
\newcommand{\Gm}{\mathbb G_{\mathrm m}}

\newcommand{\Aut}{\operatorname{Aut}}
\newcommand{\Der}{\operatorname{Der}}
\newcommand{\LND}{\operatorname{LND}}
\newcommand{\Lie}{\operatorname{Lie}}

\newcommand{\Frac}{\operatorname{Frac}}

\newcommand{\ord}{\operatorname{ord}}
\newcommand{\gr}{\operatorname{gr}}
\newcommand{\Ad}{\operatorname{Ad}}
\newcommand{\id}{\operatorname{id}}
\newcommand{\SAut}{\operatorname{SAut}}
\newcommand{\AAut}{\operatorname{AAut}}

\newcommand{\cO}{\mathcal O}

\newcommand{\widebar}[1]{\overline{#1}}

\theoremstyle{plain}
\newtheorem{theorem}{Theorem}[section]
\newtheorem{lemma}[theorem]{Lemma}
\newtheorem{proposition}[theorem]{Proposition}

\newtheorem{conjecture}[theorem]{Conjecture}

\theoremstyle{definition}
\newtheorem{definition}[theorem]{Definition}
\newtheorem{setup}[theorem]{Setup}

\theoremstyle{remark}
\newtheorem{remark}[theorem]{Remark}

\author{Alexander Perepechko}

\address{
HSE University, Faculty of Computer Science\\
Pokrovsky blvd. 11, Moscow, 109028
Russia
}

\email{a@perep.ru}

\title[Automorphism groups of (semi)rigid affine varieties]{Neutral components of automorphism groups of (semi)rigid affine varieties}
\thanks{The research was supported by the grant RSF 25-11-00302.}

\begin{document}

\begin{abstract}
We say that an affine variety $X$ is (semi)rigid if all nontrivial actions of the additive group on $X$ have the same invariant ring.
Then all such actions comprise the abelian subgroup denoted $\mathrm{SAut}(X)$.

We prove that $X$ is (semi)rigid precisely when the neutral component $\Aut^\circ(X)$ of its automorphism group is nested.
In this case, the neutral component is the semidirect product of any maximal algebraic torus and $\mathrm{SAut}(X)$.
The proof uses Laurent expansions of algebraic curves in $\Aut^\circ(X)$.
\end{abstract}

\maketitle

\section{Introduction}

Throughout the paper, $\k$ is an algebraically closed field of characteristic zero and $X$ is an integral affine $\k$-variety.
The neutral component of the automorphism group of a projective $\k$-variety is an algebraic group \cite{Mat58}, but affine automorphism groups can be infinite-dimensional: already $\Aut(\mathbb A^2)$ contains all triangular automorphisms $(x,y)\mapsto(x,y+f(x))$, where $f\in\k[x]$.
Nevertheless, $\Aut(X)$ has a natural structure of an affine ind-group \cite[Theorem~5.1.1]{FuKr18-pr}.
This leads to a basic structural problem: under what conditions is $\Aut^\circ(X)$ an algebraic group, or more generally a nested ind-group, that is, an increasing union of algebraic subgroups compatible with its ind-structure?

Our survey with Kovalenko and Zaidenberg of automorphism groups of normal affine surfaces \cite{KPZ17} led us to the following conjecture in arbitrary dimension; see \cite[Conjectures~1.1 and~1.5]{PeRe23} for precise attribution.
An element of $\Aut(X)$ is called \emph{algebraic} if it belongs to an algebraic subgroup.
We write $\Ga$ and $\Gm$ for the additive and multiplicative groups, respectively; $\SAut(X)$ denotes the subgroup generated by all $\Ga$-subgroups of $\Aut(X)$, and $\AAut(X)$ the subgroup generated by all its $\Ga$- and $\Gm$-subgroups.

\begin{conjecture}[Neutral component conjecture]\label{conj:neutral-component}
The following conditions are equivalent:
\begin{enumerate}[(i)]
  \item $\SAut(X)$ is abelian;
  \item every element of $\Aut^\circ(X)$ is algebraic;
  \item $\Aut^\circ(X)$ is nested.
\end{enumerate}
If they hold, then
\[
  \Aut^\circ(X)=T\ltimes\SAut(X)
\]
for every maximal algebraic torus $T\subseteq\Aut(X)$, and $\SAut(X)$ consists of all unipotent elements of $\Aut^\circ(X)$.
\end{conjecture}

A related, more narrowly focused conjecture suggested by Hanspeter Kraft asserts that $\Aut(X)$ is discrete whenever $X$ admits neither a nontrivial $\Ga$-action nor a nontrivial $\Gm$-action.

Several cases of the neutral component conjecture were known.
We proved it for normal affine surfaces in \cite[Theorem~1.0.3(a)]{PeZa26} and extended it to all affine surfaces in \cite[Theorem~1.4]{BePe26-pr}.
Arzhantsev and Gaifullin proved that, for rigid $X$, $\Aut(X)$ is a finite extension of its maximal torus when the corresponding grading of $\cO(X)$ is pointed, meaning that its weight cone contains no line and its degree-zero part is the base field \cite[Proposition~3.2]{ArGa17}.
Under natural conditions on normal rational affine varieties with complexity-one torus actions, Borovik and Gaifullin proved the rigid case \cite[Theorem~6.4]{BoGa25}, while Rassolov proved the semirigid case \cite{Ras26-inprep}.

In \cite[Theorem~1.3]{PeRe23}, we proved that $\AAut(X)$ is nested if and only if $\SAut(X)$ is abelian.
In this case,
\[
  \AAut(X)=T\ltimes\SAut(X)
\]
for every maximal algebraic torus $T\subseteq\Aut(X)$.
The later paper \cite[Theorem~1.1]{PeRe24} shows that $\Aut^\circ(X)$ is nested if and only if all its elements are algebraic.
Thus it remains to prove that the commutativity of $\SAut(X)$ forces every element of $\Aut^\circ(X)$ to be algebraic.

We call $X$ \emph{rigid} if it admits no nontrivial $\Ga$-action, and \emph{semirigid} if it admits such an action and all such actions have the same invariant ring.
The group $\SAut(X)$ is abelian exactly in these two cases.
The following two theorems settle the conjectures above.

\begin{theorem}[Rigid case]\label{thm:rigid}
If $X$ is rigid, then $\Aut^\circ(X)$ is an algebraic torus.
\end{theorem}

\begin{theorem}[Semirigid case]\label{thm:main}
If $X$ is semirigid, then $\Aut^\circ(X)$ is nested, and for every maximal algebraic torus $T\subseteq\Aut(X)$,
\[
  \Aut^\circ(X)=\AAut(X)=T\ltimes\SAut(X).
\]
\end{theorem}

Related results concern subgroups of $\Aut(X)$.
Kraft and Zaidenberg proved that every solvable algebraically generated subgroup is nested and is a semidirect product of a nested unipotent group and an algebraic torus \cite[Theorem~B]{KrZa24}; for a broader account, see \cite{KrZa25}.
Algebraic families generating commutative and, more recently, solvable subgroups are treated in \cite[Theorem~B]{CRX23} and \cite[Theorem~A]{CKRS26-pr}, respectively.
For the structure of connected nested subgroups and maximal commutative unipotent subgroups, see \cite{Per23-pr} and \cite[Theorem~A]{ReSa25}.

We now outline the proof.
Every element of $\Aut^\circ(X)$ lies on an irreducible curve containing the identity.
At each point of its smooth projective model, we expand the family in a local parameter as in Setup~\ref{setup:laurent-expansion}.
Motivated by gauge transformations, we differentiate the family coefficientwise and compose with its inverse to obtain a derivation; see Setup~\ref{setup:connection}.
By Lemma~\ref{lem:pole}, the obtained derivation is semisimple with integer eigenvalues at a simple pole and locally nilpotent at a pole of order greater than one.
When the curve centralizes a torus and this derivation lies in its Lie algebra, Lemma~\ref{lem:toral-correction} cancels the pole by suitable elements of the torus.
We apply this to $\cO(X)$ in the rigid case and to the associated graded algebra of the canonical filtration induced by locally nilpotent derivations in the semirigid case, where $\SAut(X)$ acts trivially.
Lemma~\ref{lem:torus-valued-curve} then globalizes these local factorizations, showing that the curve lies in $T$ in the rigid case.
In the semirigid case, Proposition~\ref{prop:graded-quotient} lifts the corresponding statement from the group quotient in the associated graded algebra to give that the curve lies in $T\ltimes\SAut(X)=\AAut(X)$.

\subsection*{Declaration of AI assistance.}
The author used OpenAI's GPT-5.6 Sol model
to generate initial drafts of the proofs of Theorems~\ref{thm:rigid} and~\ref{thm:main} in connection- and gauge-theoretic terms.
The author subsequently reconstructed these arguments in algebraic terms and rewrote the proofs for clarity.
The gauge-theoretic setup retained in the final version is introduced in Setup~\ref{setup:connection}.
The model was also used to assist with proofreading and stylistic editing, including suggestions for clearer wording and exposition.
The author independently checked all AI-assisted arguments and takes full responsibility for the correctness and originality of the manuscript.

\section{Preliminaries}\label{sec:preliminaries}

\subsection{Ind-varieties}

We briefly recall the ind-variety structure; see, e.g.,~\cite{FuKr18-pr} for details.
An ind-variety is a filtered union $V=\bigcup_iV_i$ of algebraic varieties,
where each inclusion $V_i\hookrightarrow V_{i+1}$ is a closed immersion.
A map between ind-varieties is a morphism if its restriction to every filtration element is a morphism into some filtration element of the target.
Two filtrations are \emph{equivalent} if the identity map is an isomorphism of the resulting ind-varieties.
An \emph{ind-group} is group with an ind-variety structure such that multiplication and inversion are morphisms.

Let $X$ be an affine variety and fix a closed embedding $X\hookrightarrow \mathbb A^N$.
Bounding the degrees of an automorphism and its inverse gives the standard affine ind-group structure on $\Aut(X)$; it does not depend on the embedding up to equivalence.
This structure is characterized by the following universal property:
a map $H\to\Aut(X)$ from an algebraic variety is a morphism exactly when the corresponding action map $H\times X\to X$ is a morphism \cite[Theorem~5.1.1]{FuKr18-pr}.
An ind-group is \emph{nested} if it admits an equivalent filtration by algebraic subgroups.

Every connected affine ind-group is curve-connected.
In particular, every $g\in\Aut^\circ(X)$ lies on an irreducible algebraic curve in $\Aut^\circ(X)$ that contains the identity \cite[Proposition~2.2.1 and Remark~2.2.3]{FuKr18-pr}.

\subsection{Automorphism subgroups}\label{subsec:automorphism-subgroups}

Throughout this article, $\k$ is an algebraically closed field of characteristic zero and $X$ is an integral affine $\k$-variety.
We write $\cO(X)$ for its ring of regular functions, identify $\Aut(X)$ with $\Aut_{\k\text{-alg}}(\cO(X))$, and denote by $\Aut^\circ(X)$ its neutral component, i.e., the connected component of the identity in the ind-topology.
We also denote by $\SAut(X)$ the subgroup generated by all $\Ga$-subgroups of $\Aut(X)$,
and by $\AAut(X)$ the subgroup generated by all its $\Ga$- and $\Gm$-subgroups. Both subgroups are normal in $\Aut(X)$.

For background on locally nilpotent derivations (LNDs for short), including their correspondence with $\Ga$-actions via exponentiation, see \cite{Fre17}.
Two LNDs are \emph{equivalent} if their kernels in $\cO(X)$ coincide.
The variety $X$ is \emph{semirigid} if it admits a nonzero LND and all its nonzero LNDs are equivalent; it is \emph{rigid} if it admits no nonzero LND.
Our convention, in which rigidity and semirigidity are disjoint, agrees with \cite[p.~1]{Pet16} and \cite[p.~3]{DaGu21}; the convention of \cite[Definition~1.3]{Dai20} includes rigid varieties among semirigid ones.
We say that $X$ is \emph{(semi)rigid} when either alternative holds.
Equivalently, $X$ is (semi)rigid if and only if $\SAut(X)$ is abelian; this group is nontrivial in the semirigid case and trivial in the rigid case.

For Theorems~\ref{thm:rigid} and~\ref{thm:main}, we assume that $X$ is (semi)rigid.
Under this assumption, $\AAut(X)=T\ltimes\SAut(X)$ is nested for every maximal algebraic torus $T\subseteq\Aut(X)$ by \cite[Theorem~1.3]{PeRe23}.
By \cite[Proposition~4.3]{PeRe23}, every locally finite derivation of $\cO(X)$ is a sum of an element of $\Lie(T)$ and an LND.

\begin{lemma}\label{lem:connected-normalizer}
Let $T\subseteq\Aut(X)$ be an algebraic torus.
Any curve-connected subgroup of the normalizer of $T$ centralizes it.
\end{lemma}

\begin{proof}
  Recall that the automorphism group $\Aut_{\rm grp}(T)$ of $T$ as an algebraic group is discrete.
Consider a curve $C$ in the normalizer of $T$.
The action of $C$ on $T$ by conjugations gives an algebraic family in $\Aut_{\rm grp}(T)$.
Hence the family is constant and the assertion follows.
\end{proof}

\section{The coefficientwise pole lemma}\label{sec:pole}
\begin{setup}[Coefficientwise connection]\label{setup:connection}
Let $A$ be an integral $\k$-algebra, and set
\[
  R:=\k[[z]],\qquad K:=\k((z)),\qquad
  A_R:=A\otimes_\k R,\qquad A_K:=A\otimes_\k K.
\]
Write $\partial_z=d/dz$ for coefficientwise differentiation on $A_K$.
Provided $q\in\Aut_K(A_K)$, we consider the following $\k$-linear operators on $A_K$:

  \begin{align}
    q'&:=[\partial_z,q]
    =\partial_z\circ q-q\circ\partial_z, \label{eq:commutator}\\
    \omega_q&:=-q'\circ q^{-1}, \label{eq:velocity}\\
    \nabla_q&:=\partial_z+\omega_q
    =q\circ\partial_z\circ q^{-1}. \label{eq:connection}
  \end{align}

 The operator $q'$ is $K$-linear and has the following properties:
  \begin{align*}
    q'(xy)=q'(x)q(y)+q(x)q'(y) \text { for } x,y\in A_K;\\
    (hq)'=h'\circ q+h\circ q'  \text { for } h,q\in\Aut_K(A_K).
  \end{align*}

  The operator $\omega_q$ is a $K$-linear derivation of $A_K$.  It satisfies
  \begin{align}\label{eq:velocity-identities}
    \omega_{hq}=\omega_h+\Ad_h(\omega_q) \text{ for } h\in\Aut_K(A_K).
  \end{align}

  The map $\nabla_q$ is a $\k$-linear derivation of $A_K$ extending
  $\partial_z$.  In particular, for $f\in K$ and $x,y\in A_K$,
\begin{align}
      \nabla_q(fx)=\partial_z(f) x+f\nabla_q(x),\\
      \nabla_q(xy)=\nabla_q(x)y+x\nabla_q(y).
\end{align}
Its kernel is $\ker_\k\nabla_q=q(A)$.
\end{setup}

\begin{remark}
These are standard constructions for connections and gauge transformations; see \cite[Chapter~3]{Ble81} and \cite[Chapter~II, Sections~1 and~3]{Del70}.
Indeed, under the convention $q\mathbin{\cdot}\nabla=q\nabla q^{-1}$, the operator $\nabla_q=q\partial_zq^{-1}$ is the gauge transform by $q$ of the trivial connection $\partial_z$.
In the fixed trivialization, its connection form is $\omega_q=-q'q^{-1}$, the negative right logarithmic derivative of $q$.
\end{remark}

\begin{definition}
For $0\neq x\in A_K$, its $z$-adic order is
\[
  \ord_z(x):=\max\{r\in\mathbb Z\mid x\in z^rA_R\}.
\]
An element $x\in A_R$ is \emph{primitive} if $\ord_z(x)=0$.
Suppose that $q\in\Aut_K(A_K)$ and
\[
  \omega_q(A_R)\subseteq z^{-M}A_R
\]
for some $M\geq0$.
The least such $M$ is the \emph{pole order} $m$ of $\omega_q$; it is non-negative by definition.
If $m>0$, the \emph{pole derivation} of $\omega_q$ is the derivation
\[
  D\colon A\longrightarrow A,\qquad
  D(a):=\widebar{z^m\omega_q(a)}.
\]
Here and below, the bar denotes reduction modulo $z$.
\end{definition}

\begin{lemma}[Coefficientwise pole lemma]\label{lem:pole}
Suppose that $q\in\Aut_K(A_K)$ and that $\omega_q$ has finite pole order $m$.
\begin{enumerate}[(i)]
  \item\label{item:pole-regular} If $m=0$, then $q(A_R)=A_R$.
  \item\label{item:pole-nonzero} If $m>0$, then the pole derivation $D$ of $\omega_q$ is nonzero.
  \item\label{item:pole-simple} If $m=1$, then $D$ is semisimple with integral eigenvalues.
  \item\label{item:pole-higher} If $m>1$, then $D$ is locally nilpotent.
\end{enumerate}
In particular, $D$ is locally finite if $m>0$.
\end{lemma}

\begin{proof}
Suppose first that $m=0$.
For $0\neq a\in A$, write $q(a)=z^ry$, where
  $y\in A_R$ is primitive.  Since $q(a)\in\ker_\k\nabla_q$ and
  $\omega_q$ is $R$-linear, we have
  \begin{align*}
    0=z^{1-r}\nabla_q(z^ry)
      &=r y+z\partial_z(y)+z\omega_q(y),\\
    0=\widebar{z^{1-r}\nabla_q(z^ry)}&=r\widebar y.
  \end{align*}
  Thus $r=0$, $q(A)\subseteq A_R$, and reduction modulo $z$ defines an injective algebra homomorphism $\widebar q\colon A\to A$.
  For a primitive $x\in A_R$, we have
  $\widebar{q(x)}=\widebar q(\widebar x)\neq0$, so $q$ preserves $z$-adic order on $A_K$.
  The same assertion holds for $q^{-1}$, so $q(A_R)=A_R$, proving~(\ref{item:pole-regular}).

  Now let $m>0$, and put
  \[
    P:=z^m\nabla_q=q\,(z^m\partial_z)\,q^{-1}.
  \]
  Both summands of $P=z^m\partial_z+z^m\omega_q$ do not decrease
  $z$-adic order, i.e., $P(z^rA_R)\subseteq z^rA_R$ for $r\geq0$.
   In particular, $P$ induces $D$ on $A_R/zA_R=A$, and
  \begin{equation}\label{eq:iterate-reduction}
    \widebar{P^n(a)}=D^n(a)\text{ for } a\in A,\ n\geq0.
  \end{equation}
  If $D=0$, then $z^m\omega_q(A)\subseteq zA_R$, which contradicts the minimality of $m$.
  Hence $D\neq0$, proving~(\ref{item:pole-nonzero}).

  For the remainder of the proof, fix $0\neq a\in A$ and write $y:=q^{-1}(a)=\sum_{i=1}^r v_i f_i(z)$, where $v_1,\ldots,v_r\in A$ are linearly independent, and $f_i(z)\in K$.
  Denote by $V$ the $\k$-subspace of $A$ spanned by $v_1,\ldots,v_r$.  There is an integer $c\geq0$ such that
  \begin{equation}\label{eq:bounded-order}
    \ord_z(q(x))\geq \ord_z(x)-c
    \text{ for } 0\neq x\in V\otimes_\k K.
  \end{equation}
  Indeed, take $c$ such that $q(v_i)\in z^{-c}A_R$ for every $i$.

  Suppose that $m=1$.
  Put $s=\ord_z(y)$.  Applying \eqref{eq:bounded-order} to $a=q(y)$
  gives $s\leq c$.  Since $(z\partial_z)(z^jv)=jz^jv$ for any $v\in V$, the polynomial
  \[
    F(T):=\prod_{j=s}^{c}(T-j)\in\ZZ[T]
  \]
  satisfies $F(z\partial_z)(y)\in z^{c+1}(V\otimes_\k R)$.  Therefore
  \[
    F(P)(a)=q\bigl(F(z\partial_z)(y)\bigr)\in zA_R
  \]
  by \eqref{eq:bounded-order}, and \eqref{eq:iterate-reduction} gives
  $F(D)(a)=0$. Hence
  \[
    A=\bigoplus_{r\in\mathbb Z}\ker_\k(D-r\id_A).
  \]
  This proves~(\ref{item:pole-simple}).

  Finally, suppose that $m>1$.
  The operator $\delta_m:=z^m\partial_z$ raises $z$-adic order by at
  least $m-1$.  Whenever $\delta_m^ny\neq0$, we have
  \[
    \ord_z(P^n(a))=\ord_z(q(\delta_m^ny))
    \geq \ord_z(y)+n(m-1)-c.
  \]
  For sufficiently large $n$, either $\delta_m^ny=0$ or the right-hand
  side is positive.  Thus $P^n(a)\in zA_R$, and
  \eqref{eq:iterate-reduction} gives $D^n(a)=0$.  Hence $D$ is locally
  nilpotent.  This proves~(\ref{item:pole-higher}) and completes the proof.
\end{proof}

\begin{setup}\label{setup:toral-integrality}
Let $T\subseteq\Aut(A)$ be an algebraic torus.
Let $C_K(A,T)$ be the subgroup of $\Aut_K(A_K)$ consisting of
the automorphisms that commute with every element of $T$, and let
$C_R(A,T)$ be the subgroup of its elements that preserve $A_R$.
Every member of $C_K(A,T)$ preserves the $T$-weight spaces and hence commutes with $T(K)$.
Consequently,
\[
  T(K)\,C_R(A,T)
  \subseteq C_K(A,T).
\]
This product is a subgroup; we call its members \emph{torally integral}.
For $q\in C_K(A,T)$, one has
\begin{equation}\label{eq:correction-membership}
  q\in T(K)\,C_R(A,T)
  \quad\Longleftrightarrow\quad
  bq\in\Aut_R(A_R)\text{ for some }b\in T(K).
\end{equation}
Indeed, $bq$ still centralizes $T$, and $q=b^{-1}(bq)$; the converse follows by reversing this factorization.
We use the same notation after replacing $A$ by another algebra or $(R,K)$ by $(R_p,K_p)$.
\end{setup}

\begin{lemma}\label{lem:toral-correction}
Let $q\in\Aut_K(A_K)$ centralize a faithful algebraic torus $T\subseteq\Aut(A)$.
Assume that $\omega_q$ has pole order $m\leq1$ and, if $m>0$, that the pole derivation of $\omega_q$ belongs to $\Lie T$.
Then
\[
  q\in T(K)\,C_R(A,T).
\]
\end{lemma}

\begin{proof}
If $\omega_q$ is regular, Lemma~\ref{lem:pole}(\ref{item:pole-regular}) gives $q\in C_R(A,T)\subseteq T(K)\,C_R(A,T)$.
Thus the assertion holds in this case.
Otherwise, by Lemma~\ref{lem:pole}(\ref{item:pole-simple}), the pole derivation $D$ of $\omega_q$ is semisimple with integral eigenvalues.
The weights of the faithful $T$-action on $A$ generate $X^*(T)$ \cite[Chapter~12]{Mil17}, so $D=d\lambda(1)$ for a cocharacter $\lambda\colon\Gm\to T$.
Put $b=\lambda(z)$.
For $a\in A_\chi$, put $n=\langle\chi,\lambda\rangle$.
Then $b(a)=z^na$, $D(a)=na$, and
\[
  \omega_b(a)=-b'b^{-1}(a)
  =-z^{-n}[\partial_z,b](a)
  =-z^{-n}\partial_z(z^n)a
  =-nz^{-1}a=-z^{-1}D(a).
\]
Thus $\omega_b=-z^{-1}D$.

Since both $q$ and $\partial_z$ commute with the action of $T$, they preserve its weight spaces.
Hence so does $\omega_q$, and therefore $\Ad_b(\omega_q)=\omega_q$.
The product formula \eqref{eq:velocity-identities} gives
\[
  \omega_{bq}=\omega_q-z^{-1}D,
\]
which is regular.
Lemma~\ref{lem:pole}(\ref{item:pole-regular}) gives $bq\in\Aut_R(A_R)$, and~\eqref{eq:correction-membership} gives the assertion.
\end{proof}

\begin{setup}[Laurent expansion]\label{setup:laurent-expansion}
Let $C$ be an irreducible curve and
$\gamma\colon C\to\Aut(A)$, $c\mapsto\gamma_c$, an algebraic family of
automorphisms.
For each $x\in A$, the subset $\gamma(C)(x)\subset A$ is contained in a finite-dimensional subspace of $A$.
Hence there are linearly independent $x_1,\ldots,x_s\in A$ and regular functions $f_1,\ldots,f_s$ on $C$ such that
\[
  \gamma_c(x)=\sum_{i=1}^s f_i(c)x_i.
\]
Let $\eta$ be the generic point of $C$.
Then the generic member of the family is defined by the same functions:
\[
 \gamma_\eta\in\Aut_{\k(C)}(A\otimes_\k\k(C)),
 \qquad \gamma_\eta(x)=\sum_{i=1}^s f_i x_i.
\]
It is an automorphism, since its inverse is defined by the inverse family $c\mapsto \gamma_c^{-1}$.

Let $\widetilde C$ be the smooth projective model of $C$ and fix $p\in\widetilde C$.
Write $R_p:=\widehat{\cO}_{\widetilde C,p}, \, K_p:=\Frac(R_p)$.
The canonical embedding $\k(C)=\k(\widetilde C)\hookrightarrow K_p$ induces, by scalar extension, an injective homomorphism
\[
 \Aut_{\k(C)}(A\otimes_\k\k(C)) \hookrightarrow \Aut_{K_p}(A\otimes_\k K_p).
\]
We denote by $\hat{\gamma}_p$ the image of $\gamma_\eta$ under this homomorphism.

Choose a local parameter $z$ at $p$.
Since $\widetilde C$ is smooth, $\cO_{\widetilde C,p}$ is a discrete valuation ring with uniformizer $z$.
The choice of $z$ gives identifications $ R_p\simeq\k[[z]],\,
K_p\simeq\k((z)). $ Under these identifications, the embedding $\iota_p\colon\k(C)\hookrightarrow K_p$ sends each rational function to its Laurent expansion in $z$ at $p$.
Hence, if $\gamma_\eta(x)=\sum_i f_i x_i$, then
\[
 \hat{\gamma}_p(x)=\sum_i \iota_p(f_i)x_i,
\]
and we call $\hat{\gamma}_p$ the \emph{Laurent expansion} of the family at $p$.
If $p$ lies above a point $c\in C$, then the family and its inverse are regular at $p$, so
\[
 \hat{\gamma}_p\in\Aut_{R_p}(A\otimes_\k R_p),
 \qquad \hat{\gamma}_p\bmod z=\gamma_c.
\]
\end{setup}

\begin{lemma}[No-pole principle]\label{lem:projective-coefficients}
Let $V$ be a $\k$-vector space and $\theta\in V\otimes_\k\k(C)$.
If, for every $p\in\widetilde C$, the image of $\theta$ in
$V\otimes_\k K_p$ belongs to $V\otimes_\k R_p$, then $\theta\in V$.
\end{lemma}

\begin{proof}
Write $\theta=\sum_i f_i v_i$ with linearly independent $v_i\in V$ and
$f_i\in\k(C)$.
Applying coordinate functionals dual to the $v_i$ shows that
$f_i\in R_p$ for every $p\in\widetilde C$.
Thus $f_i$ has nonnegative valuation at every $p$, so it is regular on
$\widetilde C$ and hence constant by
\cite[Chapter~I, Theorem~3.4]{Har77}.
\end{proof}

We use this lemma to give a local criterion for a curve of automorphisms to lie in a torus.

\begin{lemma}\label{lem:torus-valued-curve}
Let an algebraic torus $T$ act faithfully and rationally on $A$.
Let $\gamma\colon C\to\Aut(A)$ be an algebraic family parametrized by an
irreducible curve, assume that $\gamma(C)$ centralizes $T$, and contains the identity.
If for every $p\in\widetilde C$ one has
\[
  \hat\gamma_p\in T(K_p)\,C_{R_p}(A,T),
\]
then $\gamma(C)\subseteq T$.
\end{lemma}

\begin{proof}
Write $A=\bigoplus_{\chi\in X^*(T)}A_\chi$ for the weight decomposition w.r.t. $T$.
Fix $\chi$ with $A_\chi\neq0$, choose $0\neq w\in A_\chi$, and let $v\in A_\chi$.
The map
\[
\Gamma\colon C\to B:= A\otimes_\k A,\quad a\mapsto\gamma_a(v)\otimes_\k\gamma_a^{-1}(w)
\] is algebraic.
Let $\Gamma_\eta\in B_{\k(C)}$ be its generic value.
At $p\in \widetilde C$, write $\hat\gamma_p=t_ph_p$ with $t_p\in T(K_p)$ and $h_p\in C_{R_p}(A,T)$.
The Laurent expansion $\hat\Gamma_p\in B_{K_p}$ is
\begin{align*}
  \hat\Gamma_p
  &=\hat\gamma_p(v)\otimes_{K_p}\hat\gamma_p^{-1}(w)
  =t_ph_p(v)\otimes_{K_p}h_p^{-1}t_p^{-1}(w)\\
  &=\chi(t_p)h_p(v)\otimes_{K_p}\chi(t_p)^{-1}h_p^{-1}(w)\\
  &=h_p(v)\otimes_{K_p}h_p^{-1}(w)\in A_{K_p}\otimes_{K_p}A_{K_p}=B_{K_p}.
\end{align*}
Since $h_p^{\pm1}$ preserves $A_{R_p}$, we have
$\hat\Gamma_p\in B_{R_p}$.
  Lemma~\ref{lem:projective-coefficients} gives $\Gamma_\eta\in B$, hence $\Gamma$ is constant.
Evaluating $\Gamma$ at the identity of $\Aut(X)$, we obtain
\begin{equation}\label{eq:balanced-tensor}
  \gamma_a(v)\otimes\gamma_a^{-1}(w)=v\otimes w
  \text{ for any } a\in C.
\end{equation}

Fix $a\in C$.
Taking $v=w$ in~\eqref{eq:balanced-tensor} and comparing the nonzero pure tensors gives
\[
  \gamma_a(w)=\lambda_\chi w,
  \qquad
  \gamma_a^{-1}(w)=\lambda_\chi^{-1}w
\]
for some $\lambda_\chi\in\k^*$.
For arbitrary $v\in A_\chi$, substituting the second equality into
\eqref{eq:balanced-tensor} gives
$\gamma_a(v)\otimes w=\lambda_\chi v\otimes w$.
Since $w\neq0$, this shows that
\[
  \gamma_a|_{A_\chi}=\lambda_\chi\id_{A_\chi}.
\]
Thus $\gamma_a$ acts by homotheties on every $T$-weight space and,
by faithfulness, belongs to $T$.
Since $a$ was arbitrary, $\gamma(C)\subseteq T$.
\end{proof}

\begin{remark}
In earlier drafts, Lemma~\ref{lem:torus-valued-curve} established only that
$\gamma(C)$ is pairwise commuting; the rigid and semirigid arguments then
invoked \cite[Theorem~B]{CRX23} and \cite[Theorem~A]{CKRS26-pr}, respectively.
The present proof gives $\gamma(C)\subseteq T$ directly.
\end{remark}

\section{The rigid case}\label{sec:rigid}

In this section we assume that $X$ is rigid.
By, e.g., \cite[Theorem~2.1]{ArGa17}, $\Aut(X)$ contains a unique maximal torus, say, $T$, and $T$ is normal in $\Aut(X)$.
Lemma~\ref{lem:connected-normalizer} implies that $T$ is central in $\Aut^\circ(X)$.

\begin{proposition}\label{prop:rigid-local}
Let $C\subseteq \Aut^\circ(X)$ be an irreducible algebraic curve.
For $p\in\widetilde C$, let $\hat{\gamma}_p$ be the Laurent expansion of the inclusion
$\gamma\colon C\hookrightarrow \Aut^\circ(X)$.  Then
\[
  \hat\gamma_p\in T(K_p)\,C_{R_p}(\cO(X),T).
\]
\end{proposition}

\begin{proof}
Choose a finite set of generators of $A:=\cO(X)$, apply $\omega:=\omega_{\hat{\gamma}_p}$, and take a common denominator for their images.
Then
\[
  \omega(A\otimes R_p)\subseteq z^{-M}(A\otimes R_p)
\]
for some $M\ge0$, and let $m$ be the least such exponent.
If $m=0$, Lemma~\ref{lem:pole}(\ref{item:pole-regular}) places $\hat\gamma_p$ in $C_{R_p}(A,T)\subseteq T(K_p)\,C_{R_p}(A,T)$, and the assertion follows.

Suppose that $m>0$, and let $D$ be the pole derivation of $\omega$.
By Lemma~\ref{lem:pole}(\ref{item:pole-nonzero})--(\ref{item:pole-higher}),
$D$ is a nonzero locally finite $\k$-derivation of $A$.
If $m>1$, then $D$ is locally nilpotent, contrary to the rigidity of $X$.
Thus $m=1$, and $D$ is semisimple with integral eigenvalues and belongs to $\Lie T$.
Lemma~\ref{lem:toral-correction}, applied to $\hat\gamma_p$, gives the assertion.
\end{proof}

\begin{proof}[Proof of Theorem~\ref{thm:rigid}]
Let $g\in \Aut^\circ(X)$.
Choose an irreducible algebraic curve $C\subseteq \Aut^\circ(X)$ through
$\id_X$ and $g$.
It centralizes $T$ by Lemma~\ref{lem:connected-normalizer}, and Proposition~\ref{prop:rigid-local} gives
$\hat\gamma_p\in T(K_p)C_{R_p}(\cO(X),T)$ at every $p\in\widetilde C$.
Lemma~\ref{lem:torus-valued-curve} gives $g\in C\subseteq T$.
Conversely, $T\subseteq\Aut^\circ(X)$ because $T$ is connected.
\end{proof}

\section{The semirigid case: the graded quotient}\label{sec:degree}

For the rest of the paper, assume that $X$ is semirigid; equivalently,
$\SAut(X)$ is abelian and nontrivial.
Put $A:=\cO(X)$, choose $0\neq\delta\in\LND(A)$, and set
\begin{equation}\label{eq:filtration}
  A_{-1}=0,\qquad A_n=\ker_\k\delta^{n+1}\text{ for } n\geq0,
  \qquad B:=\gr A=\bigoplus_{n\geq0}A_n/A_{n-1}.
\end{equation}

Local nilpotence makes filtration $(A_n)$ exhaustive, and the $\delta$-degree is
additive on products of nonzero elements \cite[Proposition~1.10]{Fre17}.
Thus the filtration is multiplicative and $B$ is a domain; see
\cite[Sections~1.1 and~2.1]{Alh15-pr}.
Every nonzero $\eta\in\LND(A)$ has the form $c\delta$ with
$c\in\Frac(\ker\delta)^\times$ \cite[Principle~12]{Fre17},
so $\eta^j=c^j\delta^j$.
Hence $(A_n)$ is independent of $\delta$,
$\eta(A_n)\subseteq A_{n-1}$ for every $\eta\in\LND(A)$,
and conjugation shows that every automorphism preserves it;
see also \cite[Propositions~3.2 and~3.5]{Alh15-pr}.

\begin{proposition}[The graded quotient]\label{prop:graded-quotient}
\leavevmode
\begin{enumerate}[(i)]
  \item Every automorphism of $A$ induces a graded automorphism of $B$, giving a homomorphism
\[
  \rho\colon\Aut(A)\longrightarrow\Aut_{\mathrm{gr}}(B)
\]
whose kernel is $\SAut(X)$.
  \item Every algebraic family of automorphisms of $A$ induces an algebraic family of graded automorphisms of $B$.
  \item If $\widebar T:=\rho(T)$, then $T\to\widebar T$ is an isomorphism, $\rho(\AAut(X))=\widebar T$,
and the action of $\widebar T$ on $B$ is faithful.
\end{enumerate}
\end{proposition}

\begin{proof}
The canonical filtration defines $\rho$.
Every LND lowers the filtration, so its exponential acts trivially on $B$.
Hence $\SAut(X)\subseteq\ker\rho$.

Let $\varphi\in\ker\rho$ and put $N=\varphi-\id_A$.
Since $N(A_n)\subseteq A_{n-1}$, the operator $N$ is locally nilpotent.
By \cite[Proposition~2.57(b)]{Fre17}, $D:=\log\varphi$ is an LND and
$\varphi=\exp D\in\SAut(X)$.
Thus $\ker\rho=\SAut(X)$.

For an algebraic family $g_c$ and $a\in A_n$, the values $g_c(a)$ lie
in a finite-dimensional subspace of $A_n$ and have regular coordinates.
Projection to $A_n/A_{n-1}$ preserves this property, as does the inverse
family.  Thus the induced family of graded automorphisms of $B$ is algebraic.

Finally, $\AAut(X)=T\ltimes\SAut(X)$ identifies its image with the
faithful torus $\widebar T\simeq T$.
\end{proof}

\begin{lemma}\label{lem:central-torus}
The torus $\widebar T$ is central in $\rho(\Aut^\circ(X))$.
\end{lemma}

\begin{proof}
Let $g\in\Aut^\circ(X)$ and choose an irreducible algebraic curve
$C\subseteq\Aut^\circ(X)$ containing $g$ and $\id$.
By Proposition~\ref{prop:graded-quotient}, since $C$ normalizes $\AAut(X)$,
$\rho(C)$ normalizes $\widebar T$ and conjugation gives an algebraic family of automorphisms of $\widebar T$.
As in the proof of Lemma~\ref{lem:connected-normalizer}, this family is
constant, hence trivial at every point since it is trivial at $\id$.
Thus $\rho(g)$ centralizes $\widebar T$.
\end{proof}

\section{The semirigid case: proof}\label{sec:proof-main}
We use the Laurent expansions of Setup~\ref{setup:laurent-expansion}
and the canonical graded action from Section~\ref{sec:degree}.
We start with an elementary lifting fact for the associated graded algebra.

\begin{lemma}\label{lem:lifting}
Let $D\in\Der_\k(A)$ preserve $(A_n)$, and let $\gr(D)$ be the induced degree-zero derivation of $B$.
If $\gr(D)$ is locally finite (resp. locally nilpotent), then $D$ is locally finite (resp. locally nilpotent).
\end{lemma}

\begin{proof}
We use induction on $n$ for $a\in A_n$, starting with $A_{-1}=0$.
Local finiteness on $B$ gives a nonzero polynomial $P$ with
$P(D)a\in A_{n-1}$; induction gives a nonzero polynomial $Q$ with
$Q(D)P(D)a=0$.  Thus $D$ is locally finite.

If $\gr(D)$ is locally nilpotent, repeated application of $D$ eventually
lowers the filtration degree of each nonzero element.  Continuing in this
way gives zero, so $D$ is locally nilpotent.
\end{proof}

\begin{proposition}\label{prop:local-factor}
Let $C\subseteq \Aut^\circ(X)$ be an irreducible algebraic curve and fix $p\in\widetilde C$.
Let $\gamma\colon C\hookrightarrow\Aut^\circ(X)$ be the inclusion and
$\widebar\gamma:=\rho\circ\gamma\colon C\to\Aut(B)$ its image under the graded quotient.
Denote their Laurent expansions at $p$ by
\[
  \gamma_p\in\Aut_{K_p}(A\otimes_\k K_p),\qquad
  \widebar\gamma_p\in\Aut_{K_p}(B\otimes_\k K_p).
\]
Then
\[
  \widebar\gamma_p\in\widebar T(K_p)\,C_{R_p}(B,\widebar T).
\]
\end{proposition}

\begin{proof}
For $a\in A$, let $D_j(a)$ be the coefficient of $z^j$ in $\omega_{\gamma_p}(a)$.
Since $\omega_{\gamma_p}$ is a $K_p$-linear derivation, comparing coefficients of $z^j$ in
\[
  \omega_{\gamma_p}(ab)=\omega_{\gamma_p}(a)b+a\omega_{\gamma_p}(b)
  \qquad(a,b\in A)
\]
gives $D_j(ab)=D_j(a)b+aD_j(b)$, so each $D_j$ is a $\k$-derivation of $A$.
Laurent expansion and differentiation preserve the filtration and commute
with its quotients.  Hence each $D_j$ preserves the filtration, and
\begin{equation}\label{eq:coefficient-graded}
  \omega_{\gamma_p}=\sum_{j\geq-M}z^jD_j,
  \qquad
  \omega_{\widebar\gamma_p}=\sum_{j\geq-M}z^j\gr(D_j),
\end{equation}
The bound $M$ is obtained by taking a common denominator of a finite set of generators, as in the proof of Proposition~\ref{prop:rigid-local}.
In particular, $\omega_{\widebar\gamma_p}$ has finite pole order.
By Lemma~\ref{lem:central-torus}, $\widebar\gamma_p$ centralizes $\widebar T$.

Let $m$ be the pole order of $\omega_{\widebar\gamma_p}$.
If $m=0$, Lemma~\ref{lem:toral-correction} gives the assertion. 
Suppose $m>0$ and put $\widebar D:=\gr(D_{-m})\neq0$.
By Lemma~\ref{lem:pole}, $\widebar D$ is locally finite; hence Lemma~\ref{lem:lifting}
implies that $D_{-m}$ is locally finite.
By \cite[Proposition~4.3]{PeRe23}, recalled in Subsection~\ref{subsec:automorphism-subgroups}, we can write
$D_{-m}=\delta_s+\delta_n$ with $\delta_s\in\Lie(T)$ and $\delta_n\in\LND(A)$.
Here the summands need not commute.
Since $\delta_n$ lowers the filtration degree, $\gr(\delta_n)=0$, so
$\widebar D=\gr(\delta_s)\in\Lie(\widebar T)$ and $\widebar D$ is semisimple.
If $m>1$, Lemma~\ref{lem:pole}(\ref{item:pole-higher}) also makes $\widebar D$
locally nilpotent, contradicting $\widebar D\neq0$.
Thus $m=1$, and the residue $\widebar D\in\Lie(\widebar T)$ has integer eigenvalues by Lemma~\ref{lem:pole}(\ref{item:pole-simple}).
Lemma~\ref{lem:toral-correction} now gives the assertion.
\end{proof}

\begin{proof}[Proof of Theorem~\ref{thm:main}]
Let $g\in\Aut^\circ(X)$ and choose an irreducible algebraic curve
$C\subseteq\Aut^\circ(X)$ containing $g$ and the identity.
By Lemma~\ref{lem:central-torus}, $\rho(C)$ centralizes $\widebar T$.
Proposition~\ref{prop:local-factor} gives
$\widebar\gamma_p\in\widebar T(K_p)C_{R_p}(B,\widebar T)$ at every $p\in\widetilde C$.
Thus Lemma~\ref{lem:torus-valued-curve}, applied to $\rho|_C$, gives
$\rho(C)\subseteq\widebar T$.
Proposition~\ref{prop:graded-quotient} gives
$\rho^{-1}(\widebar T)=T\cdot\ker\rho=\AAut(X)$, so $g\in\AAut(X)$.
Since $g$ was arbitrary, $\Aut^\circ(X)\subseteq\AAut(X)$.
The reverse inclusion follows from connectedness of $\Ga$ and $\Gm$.
Finally, \cite[Theorem~1.3]{PeRe23} shows that this group is nested and that
\[
  \Aut^\circ(X)=\AAut(X)=T\ltimes\SAut(X).\qedhere
\]

\end{proof}

\printbibliography

\end{document}